\documentclass[reqno]{amsart}
\usepackage{epsfig,latexsym,amsfonts,amssymb,amsmath,amscd}
\usepackage{amsthm}
\usepackage[mathscr]{eucal}
\usepackage{mathrsfs}
\usepackage{mathbbol}
\usepackage{oldgerm,units,color}

\usepackage{pst-node}
\usepackage{graphicx}

\usepackage{eepic, epic}

\usepackage{ifthen}

\def\corn{{\operatorname{corn}}}

\usepackage{tikz}
\usepackage{epsfig,latexsym,amsfonts,amssymb,amsmath,amscd,graphics,epic}
\usepackage{amsfonts,amssymb,amsmath,amscd,amsthm}
\usepackage[mathscr]{eucal}
\usepackage{mathrsfs}
\usepackage{oldgerm,units}
\usepackage{wrapfig,epsfig}

\usepackage{ifthen}
\usepackage{mathbbol}

\usepackage{amsthm}
\usepackage[mathscr]{eucal}
\usepackage{mathrsfs}
\usepackage{mathbbol}
\usepackage{oldgerm,units}
\usepackage{wrapfig}

\newtheorem{theorem}{Theorem}[section]

\newtheorem{definition}[theorem]{Definition}

\newtheorem{example}[theorem]{Example}
\newtheorem{remark}[theorem]{Remark}

\newcommand{\Real}{\mathbb R}

\newcommand{\Net}{\mathbb N}

\newcommand{\one}{\mathbb{1}}
\newcommand{\zero}{\mathbb{0}}

\newcommand{\trop}[1]{\mathcal{#1}}

\newcommand{\tA}{\trop{A}}

\newcommand{\tD}{\trop{D}}

\newcommand{\tG}{\trop{G}}

\newcommand{\tI}{\trop{I}}

\newcommand{\tL}{\trop{L}}
\newcommand{\tM}{\trop{M}}

\newcommand{\tT}{\trop{T}}
\newcommand{\tR}{\trop{R}}

\newcommand{\tZ}{\trop{Z}}

\newcommand{\al}{\alpha}
\newcommand{\bt}{\beta}

\newcommand{\lm}{\lambda}
\newcommand{\Lm}{\Lambda}

    \ifx\proof\undefined
    \newenvironment{proof}{
    \smallskip
    \noindent\emph{Proof.}}{\hfill\(\Box\)
    \bigskip
    } \fi

\newcommand{\ifdef}[3]{\ifthenelse{\equal{#1}{true}}{#2}{#3}}

\def\({\left(}
\def\){\right)}
\definecolor{lgray}{gray}{0.90}

\def\xdtl{\dottedline{0.15}}

\def\grid{\thinlines \xdtl(1,1)(1,9) \xdtl(2,1)(2,9)
\xdtl(3,1)(3,9) \xdtl(4,1)(4,9) \xdtl(5,1)(5,9) \xdtl(6,1)(6,9)
\xdtl(7,1)(7,9) \xdtl(8,1)(8,9) \xdtl(9,1)(9,9)
\xdtl(1,1)(9,1) \xdtl(1,2)(9,2) \xdtl(1,3)(9,3)
\xdtl(1,4)(9,4) \xdtl(1,5)(9,5) \xdtl(1,6)(9,6)
\xdtl(1,7)(9,7)  \xdtl(1,8)(9,8)  \xdtl(1,9)(9,9) }

\newcommand{\ds}[1]{\ {#1} \ }
\newcommand{\dss}[1]{\quad {#1} \quad }

\def\htf{\widehat f}
\def\htg{\widehat g}
\def\hth{\widehat h}
\def\hta{\widehat a}
\def\htbfa{\widehat \bfa}

\def\mfA{\mathfrak A}
\def\mfC{\mathfrak C}

\def\mfCL{\mathfrak C(\tR)_{\operatorname{lay}}}
\def\vmap{\vartheta}

\def\mfCA{\mathfrak C(\tR)_{\operatorname{adm}}}
\def\udscr{\underline{\phantom{x}} \,}
\def\corn{{\operatorname{corn}}}

\def\admissible{admissible}
\def\algset{algebraic set}
\def\lalgset{layered algebraic set}

\def\algsets{algebraic sets}

\def\algssets{algebraic subsets}
\def\tngres{tangible part}

\def\vep{\varepsilon}

\def\lv{\operatorname{s}}

\def\Lau{\operatorname{Laur}}
\def\Ratl{\operatorname{Rat}}
\def\rat{\operatorname{rat}}

\def\mcZ{\mathcal Z}

\def\tng{{\operatorname{tng}}}

\def\mfA{\mathfrak A}

\def\mcR{\mathcal R}

\def\tSS{S}

\def\({\left(}
\def\){\right)}
\def\Z{{\mathbb Z}}
\def\Q{{\mathbb Q}}

\def\pipe{{\underset{{\ \, }}{\mid}}}

\def\vsemifield0{$\nu$-semifield$^\dagger$}
\def\vdomain0{$\nu$-domain$^\dagger$}

\def\pipe1{{\underset{{1}}{\mid}}}
\def\lmod1{\mathrel  \pipe1  \joinrel \joinrel =}

\def\Laur{\operatorname{Laur}}
\def\Pol{\operatorname{Pol}}

\def\FunSR{\operatorname{Fun} (S,R)}

\def\FunSpR{\operatorname{Fun} (S',R)}

\def\CFunFF1{\operatorname{CFun} (F,F)}

\def\semiring0{semiring$^{\dagger}$}
\def\Semiring0{Semiring$^{\dagger}$}
\def\Semirings0{Semirings$^{\dagger}$}
\def\semidomain0{semidomain$^{\dagger}$}
\def\semifield0{semifield$^{\dagger}$}
\def\semifields0{semifields$^{\dagger}$}
\def\vsemifields0{$\nu$-semifields$^{\dagger}$}
\def\domain0{domain$^{\dagger}$}
\def\predomain0{pre-domain$^{\dagger}$}
\def\predomains0{pre-domains$^{\dagger}$}
\def\domains0{domains$^{\dagger}$}
\def\vdomains0{$\nu$-domains$^{\dagger}$}

\newcommand{\xl}[2]{\,{^{[#2]}}{#1}\,}

\def\tltM{\widetilde{\tM}}

\def\Fun{\operatorname{Fun}}

\def\stTF{{\Fun_{\operatorname{abtng}}}}
\def\tGF{{\Fun_{\operatorname{gh}}}}
\def\tTFunSF{\tR_{\operatorname{tng}}}
\def\stTFunSF{\tR_{\operatorname{abtng}}}
\def\tGFunSF{\tR_{\operatorname{gh}}}
\def\domains0{domains$^\dagger$}

\def\stTFunSR{\stTF(S,R)}
\def\tGFunSR{\tGF(S,R)}
\def\corn{{\operatorname{corn}}}

\def\nucong{\cong_\nu}

\def\nug{>_\nu}

\newcommand{\etype}[1]{\renewcommand{\labelenumi}{(#1{enumi})}}
\def\eroman{\etype{\roman}}

\def\pipe{{\underset{{\tG}}{\mid}}}

\def\lmod{\mathrel  \pipe \joinrel \joinrel =}
\def\pipe{{\underset{{\tG}}{\mid}}}

\def\pSkip{\vskip 1.5mm \noindent}

\def\vmap{\vartheta}

\def\a{\alpha}
\newtheorem{thm}[theorem]{Theorem}

\newtheorem*{thm*}{Theorem}

\newtheorem{cor}[theorem]{Corollary}

\newtheorem{lem}[theorem]{Lemma}
\newtheorem{rem}[theorem]{Remark}
\newtheorem{prop*}{Proposition}

\newtheorem{prop}[theorem]{Proposition}
\newtheorem{defn}[theorem]{Definition}
\newtheorem*{examp*}{Example}
\newtheorem*{examples*}{Examples}
\newtheorem*{remark*}{Remark}
\newtheorem*{defn*}{Definition}
\newtheorem{construction}[theorem]{Construction}

\def\la{\lambda}
\def\La{\Lambda}
\def\tT{\mathcal T}
\def\Fun{\operatorname{Fun}}

\numberwithin{equation}{section}

\def\M0{M_{\zero}}

\def\supp{\operatorname{supp}}

\def\SR{R}

\def\PS{P}

\def\Cong{\Omega}

\def\rzero{\zero_\SR}

\def\mone{\one_\tM}

\def\rone{{\one_\SR}}

\def\fone{\one_F}

\def\semirings0{semirings$^\dagger$}

\newcommand{\nPS}[1]{\PS_{(!#1)}}
\newcommand{\nPSo}[1]{\nPS{\one}}

\def\srHom{\varphi}

\def\bfa{ \textbf{a}}
\def\bfb{\textbf{b}}

\begin{document}


\title[Congruences and Coordinate Semirings]
{Congruences and coordinate semirings \\[2mm] of tropical varieties}


\author[Z. Izhakian]{Zur Izhakian}
\address{Institute  of Mathematics,
 University of Aberdeen, AB24 3UE,
Aberdeen,  UK}
\email{zzur@abdn.ac.uk; zzur@math.biu.ac.il}

\author[L. Rowen]{Louis Rowen}
\address{Department of Mathematics, Bar-Ilan University, Ramat-Gan 52900,
Israel} \email{rowen@math.biu.ac.il}

\subjclass[2010]{Primary   14T05, 16Y60   16Y60; Secondary 06F20,
12K10.
  }

\date{\today}


\keywords{Tropical algebra, tropical geometry, supertropical
algebra,  $\nu$-domain, $\nu$-semifield, coordinate \semiring0,
admissible variety,  dimension.}


\thanks{\noindent \underline{\hskip 3cm } \\ File name: \jobname}


\begin{abstract}

In this paper we present two intrinsic algebraic
 definitions of tropical variety motivated by the classical
 Zariski correspondence, one utilizing the algebraic
structure of the coordinate \semiring0 of an affine
  supertropical \algset, and the second based on the layered
  structure.  We tie them to tropical
geometry, especially in connection with the dimension of an affine
variety.
\end{abstract}

\maketitle





\section{Introduction}
\numberwithin{equation}{section} The goal of this paper is to study
families of affine supertropical varieties in terms of their
coordinate \semirings0, or equivalently certain congruences of the
polynomial semiring, paying particular attention to an algebraic
formulation of tropical dimension which will match the intuitive
definition obtained from simplicial complexes. Tropical varieties
have been the focus of much investigation in tropical geometry,
cf.~\cite{Gat,IMS}, often defined in terms of polyhedral complexes (i.e., piecewise linear
objects) satisfying the balancing condition, but this approach,
although successful for curves and hypersurfaces, is not fully compatible  with
Zariski's approach to viewing varieties as the zero locus of an
ideal $\mathcal I$ of polynomials in $K[\la_1, \dots, \la _n]$ over
a field $K$. A key feature of Zariski's approach is the prime
spectrum of the coordinate ring $K[\la_1, \dots, \la
_n]/\mathcal I$, which, in classical theory, also is identified with the algebra of
polynomials restricted to the variety.
%

The authors have translated the tropical theory to an algebraic language more
amenable to structure theory, for example in \cite{zur05TropicalAlgebra},
\cite{IzhakianRowen2007SuperTropical},
\cite{IzhakianKnebuschRowen2009Refined}, \cite{IKR4}, and
\cite{IzhakianKnebuschRowen2011CategoriesII}, where an extra ``ghost
level'' $\tA^\nu$ is adjoined to the original max-plus algebra~$\tA$,
and additive idempotence is replaced by supertropicality, i.e., $a+a = a^\nu,$
cf.~\S\ref{trop}. 
In this framework, the \algset\ of a collection of polynomials is
just the set of vectors all taking on ghost values. Although
encapsulating the definition of ``corner locus'' in standard
tropical geometry, this approach enables one to set up a direct
algebraic approach analogous to the Zariski correspondence.

Even so, one encounters difficulty when considering algebraic sets
of polynomials: The intersection of tropical varieties need not be a
tropical variety in the usual sense (even for planar curves). For
example, the non-transversal intersection of the curves defined by
$x+y+0$ and $x+y^2+0$ is the union of the two rays emanating along
the axes from the origin and fails the balancing condition, as does
the non-transversal intersection of the lines defined by $x+y+0$ and
$x+y +1,$ cf.~Figure~\ref{fig:0} (a) and (b), respectively. Such
curves can be excluded via a requirement that curves are in generic
position, but one would prefer a theory that   deals with all cases.

\begin{figure}[h]
\begin{pspicture}(5,5)(0,0)
\setlength{\unitlength}{0.5cm}
\grid
\Thicklines
\put(1.5,5){\line(1,0){3.5}}
\Thicklines
\put(5,1.5){\line(0,1){3.5}}
\Thicklines
\put(5,5){\line(1,1){3.5}}
{\red \put(5,5){\line(2,1){3.5}}}
\put(4,0){(a)}
\end{pspicture}
\begin{pspicture}(5,5)(0,0)
\setlength{\unitlength}{0.5cm}
\grid
\Thicklines
\put(1.5,5){\line(1,0){3.5}}
\put(5,1.5){\line(0,1){3.5}}
\put(5,5){\line(1,1){3.5}}
{\red \put(1.5,6){\line(1,0){4.5}}}
{\red \put(5.8,1.5){\line(0,1){4.5}}}
\put(4,0){(b)}
\end{pspicture}

\caption{}\label{fig:0}
\end{figure}

The approach taken in this paper is to define the \textbf{coordinate
\semiring0} of an \algset\ $X$ as the \semiring0 of polynomials
(over the supertropical structure $ \tA \cup \tA^\nu$), realized as
functions, restricted to $X$. This enables one to study the spectrum
(but now of congruences rather than ideals), and leads to a
correspondence between \algsets\ and congruences, as indicated in
\cite{IKR4}. The challenge remains of using the algebraic structure
to filter out the ``bad'' \algsets\ of the previous paragraph,
namely, those that do not satisfy the balancing condition. The
obvious way is to restrict the class of permissible congruences
defining our algebraic sets. Several options have been proffered,
most notably the ``bend congruences'' of \cite{GG}. In this paper,
we present two supertropical alternatives which are based on
algebraic and topological considerations.

Our main approach, given in \S\ref{coor}, is via the coordinate
\semiring0 of Definition~\ref{coord2}. We impose a requirement on
functions whose value on a dense algebraic subset are equal, and
call these algebraic sets \textbf{admissible}.  This natural
condition is automatic in the classical algebraic geometrical world,
by virtue of the easy part of the fundamental theorem of algebra,
but needs to be stipulated in the tropical world. Tropical
hypersurfaces are admissible, by Proposition~\ref{admis}, whereas
when we ruin the balancing condition by erasing a facet, the
algebraic set becomes inadmissible, by Proposition~\ref{admp}. In
this way, admissibility provides a natural generalization of the
balancing condition in higher codimensions.

Once one focuses on the appropriate algebraic sets, it is not
difficult to define the dimension in terms of the length of chains
of admissible varieties in \S\ref{admiss1}, and prove that it is
well-defined
 and consistent with the geometric
intuition (Theorem~\ref{primelen2}). Nevertheless, at times the
theory diverges from classical algebraic geometry. For example,
\algsets\ can decompose non-uniquely into varieties, as is seen
in~Example~\ref{irredundant}.



The main weakness of Definition~\ref{coord2} comes from its
strength: The intersection of admissible algebraic sets need not be
admissible.  Indeed, in the planar scenario,  we do not want the
intersection of a tropical line and quadric to be admitted, since
then we would have to permit all line segments as varieties, and
thus they all would be reducible (except for the points). On the
other hand, if one wants to define a topology whose base is the
closed sets, one needs the intersection of varieties to be a
variety. This leads us in our second approach to a further
refinement of the supertropical structure, namely the layered
structure of \cite{IzhakianKnebuschRowen2009Refined}, and in
\S\ref{sec:layer} we present a   class of congruences which is
closed under intersections, taken the layering into account, and
also is Noetherian by Proposition~\ref{Noeth}. Thus, we also have a
notion of dimension here, but globally it is larger than the
simplicial dimension. This discrepancy can be overcome, but requires
a more detailed local treatment that is beyond the scope of this
paper.

\section{Background}

We review a few notions from semigroups and semirings.
As customary, $\Net$ denotes the positive natural numbers,  $\Q$ denotes the rational numbers, and
$\Real$ denotes the real numbers.
\subsection{Semigroups and monoids} $  $

A \textbf{monoid} is a semigroup with a unit element $\mone$.  For
any semigroup $\tM := (\tM, \cdot \,)$ we can formally adjoin the
unit element $\mone$ by declaring that $\mone a = a\mone = a$ for
all $a\in \tM,$ so when dealing with multiplication we work with
monoids.

 An Abelian monoid
$\tM := (\tM, \cdot \, )$ is \textbf{cancellative} with respect to a
subset $S \subseteq \tM$ if $as = bs$ implies $a=b$ whenever $a,b
\in \tM$ and $s\in S.$ In this case, we also say that $S$ is a
\textbf{cancellative} subset of $\tM$.
%

\subsection{Ordered monoids}

\begin{defn}
 A \textbf{partially ordered monoid}
is a monoid~$\tM $ with a partial order satisfying
\begin{equation}\label{ogr1} a \le b \quad \text{implies}\quad ca
\le cb,\end{equation} for all elements $a,b,c \in \tM$. A monoid
$\tM $ is   \textbf{ordered} if the order is total.
\end{defn}

Note that this definition excludes ordered Abelian groups such as
$(\Q,\cdot \,)$ from consideration; on the other hand, $(\Q,+\,)$ is
ordered in this sense.
%

 \begin{defn}
 A
semigroup $\tM := (\tM, \cdot \,)$ is called $\Net$-\textbf{divisible}  if $\root n
\of a \in \tM$ for all $a\in \tM$ and all $n \in \mathbb N.$
A
monoid is \textbf{power-cancellative} if $a^m = b^m$ implies
$a=b$.
\end{defn}

 \begin{rem} One can uniquely define rational powers of any element in an $\Net$-divisible, power-cancellative semigroup $\tM$; adjoining a
unit element $\one_\tM$ to $\tM$, we could define $a^0 = \one_\tM$.
\end{rem}

\begin{rem}\label{divcl00}  By Bourbaki~\cite{B},
any strictly  cancellative Abelian  monoid $\tM$ can be embedded
into an $\Net$-divisible Abelian monoid $\tltM$, which we call the
\textbf{divisible closure} of $\tM$. Namely, by passing to the group
of fractions, cf.~ \cite{B}, we may assume that $\tM$ is a group. We
formally introduce $\root m \of a$ for each $a\in \tM$, identifying
$\root m \of a$ with $\root n \of b $ iff $a^n = b^m.$ We define the
product
$$\root m \of a \root n \of b = \root {mn} \of {a^nb^m}.$$

\end{rem}

%
%

 \begin{lem}\label{divcl01} If $\tM$ is partially ordered, then $\tltM$ is endowed with the
partial order given by
  $$\root m \of a \le \root n \of b  \dss{\text{iff}} a^n \le
b^m.$$ If $\tM$ is power-cancellative, then $\tltM$ is
power-cancellative.\end{lem}

 \begin{proof} The relation is   well-defined, and is easily
 seen to be a partial order. Furthermore, if $(\root m \of a)^k = ( \root n \of b)^k, $
 then $a^{nk} = b^{mk},$ implying $a^n = b^m,$ and thus $ \root m \of a =   \root n \of
 b .$
\end{proof}

In summary, any cancellative, power-cancellative ordered Abelian monoid
can be embedded into an $\Net$-divisible,  power-cancellative
ordered Abelian group, so we usually assume these hypotheses.

\subsection{\Semirings0}$ $

Semirings were studied by Costa~\cite{Cos}.
 A standard general reference for the structure of
semirings is~\cite{golan92}.  For reasons   discussed   in the
introduction of \cite{IKR4}, it is convenient to deal a semiring
without a zero element, which we call a \textbf{\semiring0}. Thus,
a \semiring0 $(R,+ \, ,\cdot \;, \rone)$ is a set $R$ equipped with two
binary operations $+$ and~$\cdot \;$, called addition and
multiplication, such that:
\begin{enumerate}
    \item $(\SR, +)$ is an Abelian semigroup; \pSkip
    \item $(\SR, \cdot \;, \rone )$ is a monoid with identity element
    $\rone$; \pSkip
    \item Multiplication distributes over addition; \pSkip
    \item There exist $a,b \in R$ such that $a+b = \rone.$\pSkip
\end{enumerate}

Condition (4) is a very weak condition that we do not need in this
paper, but is needed to develop the theory of modules in later work.
It is automatic in semirings with zero since $\rone + \rzero =
\rone,$ and also is obvious in the max-plus algebra since $\rone + b
= \rone$ for any $b \le \rone.$

\begin{rem}\label{makesemi} Any  ordered monoid $(\tM, \cdot \; )$
  gives rise to
a \semiring0, where we define $a+b$ to be $\max\{ a, b\}.$ Indeed,
associativity is clear, and distributivity  follows from
\eqref{ogr1}.
\end{rem}

 One can always adjoin an additive neutral element $\rzero$ to a
\semiring0 to get a semiring, via the multiplicative rule $$\rzero
\cdot a = a \cdot \rzero
 =\rzero \qquad \forall a\in R.$$
%

%

\begin{definition}
A \textbf{homomorphism} of \semirings0 is defined as a function
$\srHom : R \to R' $  that preserves addition and multiplication.
To wit, $\srHom  $   satisfies the following properties for all
$a$ and $b$ in $R$:
\begin{enumerate}
    \item $\srHom(a + b) = \srHom(a) + \srHom(b)$; \pSkip
    \item $\srHom(a \cdot b) = \srHom(a) \cdot \srHom(b)$;\pSkip
    \item  $\srHom(\rone) = \one_{R'}$. \pSkip
\end{enumerate}
\end{definition}
%

%

The structure theory of \semirings0 is motivated by general
considerations on universal algebra, for which we use \cite{Jac1980}
as a reference. We  recall as a special case from \cite
[p.~61]{Jac1980} that a \textbf{congruence} $\Cong$ on a \semiring0
$R$ is an equivalence relation $\equiv$ preserving addition and
multiplication, i.e., if $a _i \equiv b_i$ then $a_1+a_2 \equiv
b_1+b_2$ and $a_1a_2 \equiv b_1b_2$. Sometimes we denote $\Cong$ as
the relation $\equiv$, or, equivalently, as $\{ (a,b): a \equiv b\}
$, a sub-\semiring0 of $R\times R.$

 A congruence $\Cong $ is  \textbf{cancellative} if $ca \equiv cb$ implies  $ a
  \equiv b$; $\Cong$ is
\textbf{power-cancellative} when $R/\Cong$  is power-cancellative
(as a multiplicative monoid), i.e., if
 $a_1 ^k  \equiv  a_2^k$ for some $k \ge 1$ then $a_1 \equiv
a_2$. (Power-cancellative congruences, also called torsion-free in
\cite{CHWW}, play the role of radical ideals.)

Any \semiring0 homomorphism $\varphi: R \to R'$ gives rise to a
congruence $\Cong_\varphi$ on $R$ given by $(a,b) \in \Cong_\varphi $ iff $
\varphi(a) = \varphi(b)$;  conversely, any congruence $\Cong$ gives
rise to a  \semiring0 structure  $R/\Cong$ on the equivalence
classes, and a natural homomorphism $\varphi: R \to R/\Cong$ given
by $a \mapsto [a].$

\begin{example}
We define the \textbf{trivial congruence} $\Cong = \{ (a,a): a \in R
\}$; in this case, $\varphi: R \to R/\Cong$ is an isomorphism.
 \end{example}

As we shall see, the family of all congruences on the supertropical
structure is too broad to support a viable geometric theory, so we
restrict the family, to be specified later.

Let $\mfC (R)$ denote a given family of congruences on a given \semiring0~$R$.

 \begin{defn}\label{def:irrdCong}   A  congruence $\Cong\in \mfC (R)$ is
$\mfC (R)$-\textbf{irreducible}   if it cannot  be written as an
intersection $\Cong_1 \cap \Cong_2$ of congruences $\Cong_1$ and
$\Cong_2$ in $\mfC (R)$, each properly containing $\Cong.$
\end{defn}

 $\mfC (R)$ is too broad for our purposes without
a serious restriction. We say that $\mfC (R)$ is \textbf{Noetherian}
if any ascending chain of congruences in $\mfC (R)$ terminates.
Equivalently, any subset of congruences in  $\mfC (R)$ has a maximal
member. (For example, in classical algebra, one often takes $\mfC
(R)$ to be the finitely generated congruences of the polynomial
algebra.) The following observation is a standard application of
Noetherian induction:

\begin{prop} Every congruence in a Noetherian family $\mfC (R)$ of
congruences
 is a
finite intersection of  $\mfC (R)$-irreducible congruences.
\end{prop}
\begin{proof} Any maximal counterexample would be the intersection
of two larger congruences in $\mfC (R)$, each of which by hypothesis is
a finite intersection of congruences that are $\mfC (R)$-irreducible.
\end{proof}

There are several candidates for a working definition of $\mfC (R)$,
such as \cite{MR}. In this paper we offer two: First, a traditional
one using the Zariski topology, in \S\ref{supertrop1}, and then one
in terms of the layered theory given in \S\ref{lay1}.
%
\subsection{\vdomains0}\label{trop}$ $

Despite the elegance of Remark~\ref{makesemi}, the structure of the
resulting \semiring0 is too crude for some algebraic applications.
 To remedy this, we recall briefly the basics of supertropical algebra and generalize them in order
 to be able to handle functions.

\begin{defn}\label{super1} A \textbf{$\nu$-\semiring0} is a quadruple
$R := (R, \tT, \tG, \nu)$ where  $R$ is a \semiring0, $\tT \subset
R$ is a multiplicative submonoid, $\tG \subset R$ is a partially ordered
\semiring0 ideal, together with a map $\nu: R \to \tG$, satisfying
$\nu^2 = \nu$ as well as the conditions:
$$ \begin{array}{lll}
 a+b =  a  &   \text{whenever}  &  \nu(a) > \nu(b), \\
 a+b = \nu(a) &  \text{whenever} & \nu(a) = \nu(b).
\end{array}$$

$R $ is called a \textbf{$\nu$-\domain0} when the multiplicative
monoid $(R, \cdot \, )$ is  commutative and cancellative with
respect to $\tT$.

If furthermore $\tT$ (and thus also $\tG$) is an Abelian group, we
call $R$ a \textbf{\vsemifield0}.
\end{defn}
%

\noindent We write~$a^\nu$ for $\nu(a).$ We write $a \nucong b$
whenever $a^\nu = b^\nu$, and $a >_\nu b$ (resp. $a \geq_\nu b$) whenever $a^\nu > b^\nu$(resp. $a^\nu \geq b^\nu$).

$\tT$ is called the monoid of   \textbf{tangible elements}, while
the elements of $\tG$ are called \textbf{ghost elements} and $\nu: R
\to \tG$ is called the \textbf{ghost map}. Intuitively, the ghost
elements in $\tG$ correspond to the original max-plus algebra, and
$R$ is a cover of $\tG$.   But our interest lies in the tangible
layer $\tT$, since it captures the tropical geometry.

\begin{defn}\label{super2} A \textbf{supertropical
\domain0} is a $\nu$-\domain0 $R := (R, \tT, \tG, \nu)$ for which
$\tG := R \setminus \tT $ is ordered and the restriction $\nu|_\tT: \tT \to \tG$ is onto.
 If, moreover, $\tT$ is an Abelian group, we
call $R$ a \textbf{supertropical \semifield0}.
 \end{defn}

For each $a$ in a supertropical \domain0
 $R$ we choose
 an element $\hta \in \tT$ such that ${{\hta}}^\nu = a^\nu.$
 (Thus $a \mapsto \hta$ defines a section from $\tG$ to $\tT,$
 which we call the \textbf{tangible lift}.)
Likewise, for $\bfa = (a_1, \dots, a_n) \in R^{(n)},$ we define its
\textbf{tangible lift} $\htbfa := (\hta_1, \dots, \hta_n).$

 To clarify our exposition,  most of the examples in this
paper are presented for the supertropical \semifield0 $ (\Q \cup \Q
^\nu,  \Q, \Q ^\nu, \nu)$, where  $\one := 0_\Q$, cf.~
\cite{zur05TropicalAlgebra} and
\cite{IzhakianRowen2007SuperTropical}, built from the ordered group
$(\Q,+)$, whose operations are induced by the standard operations
$\max$ and $+$. Here, $\tT$ is one copy of $\Q$ whereas $\tG =
\Q^\nu,$ another copy of $\Q$, and $\nu |_{\tT}:  {\tT} \to \tG$ is
an isomorphism; hence we can take the tangible lift simply to be
$(\nu |_{\tT})^{-1}$. Likewise, the same construction could be for
any ordered Abelian group instead of $(\Q, +).$


 Tropical geometry is deeply connected to simplicial
complexes, and we also need the relevant topology in this setting.
\begin{defn}
The $\nu$-\textbf{topology} on a
supertropical \semifield0 $F$ is defined as having the sub-base of
neighborhoods $B(a,\vep) : = \big\{b \in F: \frac b a \le_\nu \vep
\big\},$ where $a,\vep \in \tT.$
\end{defn}

 But we work in the generality of Definition~\ref{super1} in
order to handle functions, in particular polynomials, for which
 $\tG$ is   only partially ordered.
Our structure of choice for understanding tropical geometry is the
polynomial \semiring0 over the $\nu$-\semifield0.

 We want to
describe congruences that arise with the $\nu$-structure.

\begin{rem}
Any congruence on  a $\nu$-\semiring0 $R$ satisfies the condition
that if  $a \equiv b$ then $$a^\nu = \rone^\nu a \equiv \rone^\nu b =
b^\nu.$$
\end{rem}
%

\begin{rem} If $\Cong$ is a
congruence of a $\nu$-\domain0 $R:= (R, \tT, \tG, \nu)$, then $\nu$
induces a ghost map $[\nu]$ on $R/\Cong = (R, \tT/\Cong,
 \tG/\Cong, [\nu])$ via $[a]^{[\nu]} = [a^\nu],$ and  when $\nu|_\tT \to \tG$ is 1:1, then the restriction $[\nu]: \tT/\Cong
\to \tG/\Cong$ also is 1:1. \end{rem}

We are interested in those congruences that yield $\nu$-domains.
Towards this end, we have:

\begin{defn}\label{def:primeCOng}  A congruence $\Cong$ on a $\nu$-\semiring0 $R$ is
\textbf{tangibly cancellative} when $ca \equiv cb$ implies $a \equiv b$ for any  $a \in \tT$.   \end{defn}


%
%

%
\section{Polynomial  \semirings0 over supertropical \domains0}

Our main strategy is to define affine tropical varieties in terms of
polynomials. We treat polynomials as functions that are defined
logically as elementary sentences, and study their algebraic
structure as a \semiring0.

\subsection{The function monoid and \semiring0}

\begin{defn}\label{Fun1} Given a monoid $\tM := (\tM, \cdot \, ),$ we define the monoid of functions $\Fun (S, \tM)$ to be the set-theoretic functions from
$S$ to~$\tM$, in the usual way (via pointwise
multiplication).\end{defn}

We say that a function $g \in \Fun (S, \tM) $ \textbf{dominates} a function $f \in \Fun (S, \tM)$ at $\bfa$ if $f(\bfa) \le g(\bfa).$
We take the   corresponding partial order on $\Fun (S,\tM)$ given
by
$f\le g$ iff $f(\bfa) \le g(\bfa)$ for each $\bfa \in S.$ 

\begin{lem}\label{cancext} If the monoid $\tM$ is  cancellative,
then so is $\Fun (S, \tM)$.\end{lem}
\begin{proof} Easy componentwise verification, given in
\cite[Lemma~7.3]{IKR4}.\end{proof}

\begin{rem}\label{Fun10} When $\tM$  is a \semiring0, then $\Fun (S, \tM)$ is also a \semiring0  in the usual way
(via pointwise addition).\end{rem}

As customary, we write $f|_U$  for the restriction of a function $f \in \Fun (S, \tM)$ to a nonempty subset $U \subset S.$
Although failing to satisfy bipotence,  $
\FunSR$ does satisfy the weaker property for a \semiring0~$R$:

\medskip
\begin{rem}
 If $R$ is idempotent then so is $ \FunSR$, as seen by pointwise verification.
 \end{rem}

\begin{rem}\label{Fun2} Given any sets $S' \subseteq S$, there is
a natural onto homomorphism $\FunSR \to \FunSpR$ given by $f\mapsto
f|{_{S'}}.$
%
Our main interest in this paper is to study chains of these
homomorphisms.
For any homomorphism $\varphi: R \to R'$ and $\bfa
\in S,$ we can define the \textbf{evaluation homomorphism}
$$\psi_{\bfa,\varphi}:\FunSR \to R', \qquad f \mapsto
\varphi(f(\bfa)).$$ 
\end{rem}

The point of using $\nu$-domains is in the following observation:

\begin{rem}\label{functan} Given a \vdomain0 $R := (R, \tT, \tG, \nu)$, define
\begin{equation}\label{eq:abstang}
\begin{array}{rcl}
\tGFunSR &:= &\{ f\in \FunSR \ds: f(\bfa) \in \tG \text{ for all } \bfa \in S \}, \\[2mm] \stTFunSR &: = &\{ f \in \FunSR \ds: f(\bfa) \in \tT \text{ for all } \bfa \in S\}.
\end{array}
\end{equation}
 Then $(\FunSR, \stTFunSR, \tGFunSR, \nu)$
becomes a $\nu$-domain, the main object of this paper, where we
define $f^\nu$ by $f^\nu(\bfa) := f(\bfa)^\nu$. If $R$ is a
supertropical \domain0, then so is $\FunSR$, since $\nu$ induces an
onto map $ \stTFunSR\to  \tGFunSR$ .
\end{rem}

\begin{example}  The functions
of interest to us are the \textbf{polynomials}  in $\Lm := \{ \lm_1,
\dots, \lm_n\}$, defined by formulas in the elementary language
under consideration. $R[\La]$ denotes the usual polynomials over the
\semiring0 $R$. In our examples, $\rone = 0$, so we write $\la$ for
$0 \la.$

 If we adjoin the symbol $^{-1}$ (for multiplicative inverse), then
 we have the \textbf{Laurent polynomials}
  $R[\La^\pm]: = R [\lm_1, \dots, \lm_n, \lm_1^{-1}, \dots, \lm_n^{-1}
].$ If our language also includes the symbol $
\sqrt[m]{\phantom{w}},$ i.e., if we are working over a
power-cancellative, divisibly closed monoid, then we may consider
the polynomials $R[\La]_{\rat}$ with rational powers.

 We need to
study polynomials (in the appropriate context) and their roots, but
viewed in the above context as functions under the natural map given
by sending a polynomial $f$ to the function $\bfa \mapsto f(\bfa).$
 Thus, $\Pol (S,R)$ denotes the image in   $\Fun (S,R)$ of
  $R[\La]$,
 $\Lau (S,R)$ denotes the image of $R[\La^\pm],$ and $\Ratl (S,R)$ denotes
the image of  $R[\La]_{\rat}$.
  \end{example}

When $R$ is a supertropical \domain0, $\Pol (S,R)$, $\Lau (S,R)$,
and $\Ratl (S,R)$ are sub-$\nu$-\domains0 of $\Fun(S,R)$. (But their
$\nu$-structure differs from that of $\Fun(S,R)$  because of the
issue of tangibility, as we shall see.)

\subsection{Decompositions of polynomials} $ $

\emph{We assume throughout the remainder of this paper that $F =(F,
\tT, \tG, \nu)$ is a supertropical-\semifield0, with $\nu$ 1:1 and
onto, and we have a given tangible lift $\tG \to \tT$ given by
$\nu^{-1}$, and $S\subset F^{(n)}$ is given. $\tR$ denotes
$\Pol(S,F)$, $\Lau(S,F)$, or $\Ratl(S,F)$, and monomials and
polynomials are taken in the appropriate context. } Namely any
monomial has the form $h = \a \lm_1^{i_1}\cdots \lm_n^{i_n}$ for $\a
\in \tR$ and each $i_j$ in $ \Net,$ $\Z$,  or $\Q$ respectively . We
call $\lm_1^{i_1}\cdots \lm_n^{i_n}$ the \textbf{pure part} of $h$.
Note that if $h_1 = \a_1 \lm_1^{i_1}\cdots \lm_n^{i_n}$ and $h_2 =
\a _2 \lm_1^{i_1}\cdots \lm_n^{i_n}$ have the same pure part, then
$h_1 + h_2 = (\a_1 + \a_2)\lm_1^{i_1}\cdots \lm_n^{i_n}$ is also a
monomial.


\begin{rem}\label{Lau} Customarily one takes $\tR = \Pol(S,F)$, but  it is easy to check via localization at the $\la _i$ that the
definitions provide the same results for $\tR = \Lau(S,F)$.\end{rem}

\begin{defn}
 A \textbf{decomposition} of $f\in \tR$ is a sum $f := \sum_i h_i$
of monomials whose pure parts are distinct. (In other words, the number of monomials that
are summands of $f$ is minimal.)   The monomial~ $h_i$  is \textbf{essential} in $f$ at
$\bfa$   if $f|_U \neq (\sum_{j \neq i} h_j)|_U$ for some open neighborhood $U$ of $\bfa$.
A monomial $h_i$  is \textbf{essential} in~$f$ if it is {essential} in~$f$ at~ $\bfa$ for some point $    \bfa $.
  A polynomial is \textbf{essential}  if each monomial in its
  decomposition is essential.
 \end{defn}

 Thus, a  polynomial  $f$
is a tangible monomial iff it has no proper decomposition. (In
fact, this is an intrinsic way to define monomial.) We also need
to handle the case in which a monomial is not essential anywhere,
but does contribute to $f$ by taking on the same value at some
point.

\begin{defn} Decomposing a polynomial $f :=
\sum_i h_i$ as a sum of monomials, we say that an inessential
monomial $h_i$ of $f$ is \textbf{quasi-essential} at $\bfa$ if
$f(\bfa) \nucong h_i(\bfa).$  An inessential monomial $h_i$  is
\textbf{quasi-essential} in $f$ if it is quasi-essential in $f$ at
$\bfa$ for some point $    \bfa$.

The  \textbf{support}~ $\supp_{\bfa}(f)$ of $f = \sum_i h_i$ at
the point $\bfa \in S$
 is the set of monomials $h_i$ which  dominate $f$ at $\bfa$.  The  \textbf{support}~ $\supp(f)$ of $f$ is $\bigcup _{\bfa \in S}
 \supp_{\bfa}(f).$

The
\textbf{shell} of the decomposition of $f$ is the sum of the
essential monomials $h_i$
 in $\supp(f)$.
\end{defn}

\begin{example} The polynomial $f = \la ^2 + 6$ has the obvious decomposition as
written,  and is its own shell.
For the polynomial $f = \la ^2 + 3\la+ 6,$  the monomial $h = 3
\la$ is quasi-essential, since $f(3) =   6^\nu$, whereas $h(3) =
 6$. 
  \end{example}

  \begin{example} The polynomial $g =2 \la^2 _1 +2 \la _2^2 + 0$  is the
  shell of~$f = 2\la^2 _1 + 2\la _2^2 + \la_1 \la_2 +0$, because $\la_1\la_2$ is dominated by $2\la_1^2 + 2\la_2^2.$
  \end{example}

 Example~\ref{badshell} below shows  how a monomial can be
quasi-essential at one point  but essential somewhere else.
%

\begin{lem}\label{multi1}
Any monomial $h$ is multiplicative along any line, in the sense
that
$$h(\bfa^t \bfb^{1-t} ) = h(\bfa)^t h(\bfb)^{1-t}$$
 for all $t \in \Real$.
\end{lem}
\begin{proof} Write $h = \a \la_1^{i_1}\cdots \la_n^{i_n},$ $\bfa = (a_1,\dots, a_n)$, and $\bfb = (b_1,\dots, b_n)$.  Then
$$h(\bfa^t \bfb^{1-t} )
=\a (a_1^{t}b_1^{1-t})^{i_1}\cdots (a_n^{t} b_n^{1-t})^{i_n} =
\a^t a_1^{i_1t}\cdots a_n^{i_nt}\ \a^{1-t} b_1^{i_1(1-t)}\cdots
b_n^{i_n(1-t)} =
 h(\bfa)^t h(\bfb)^{1-t}.$$
\end{proof}

\begin{prop}\label{multi12} If two monomials $h_1$ and $h_2$ are equal at
two points $\bfa$ and $\bfb$  then they are equal
at every point in the line connecting $\bfa$ and $\bfb$.\end{prop}
\begin{proof} Follows at once from the lemma.
\end{proof}

\begin{prop}\label{multi2} If a monomial $h_1$ dominates $h_2$ at
two points $\bfa$ and $\bfb$ then $h_1$ dominates $h_2$ at every
point in the line connecting $\bfa$ and $\bfb$.\end{prop}
\begin{proof} Each point can  be written as $\bfa^t \bfb^{1-t}$
for $ 0 \le t \le 1,$ and so
$$h_1(\bfa^t \bfb^{1-t} )
=h_1(\bfa)^t h_1(\bfb)^{1-t}\ge_\nu h_2(\bfa)^t h_2(\bfb)^{1-t} =
h_2(\bfa^t \bfb^{1-t} ).$$
\end{proof}

\begin{defn}\label{tanright}  A polynomial $f \in \tR$
is \textbf{tangible} when all of the coefficients of its essential
monomials are tangible. $ \tTFunSF$ denotes the monoid of tangible
polynomials, and  $  \tGFunSF$ denotes the ideal of polynomials
whose essential monomials have ghost coefficients.
\end{defn}

\begin{rem}\label{tang2} This does not quite
match the definition of $\stTFunSR$ in Remark~\ref{functan},
cf.~\eqref{eq:abstang}. For example, taking $f = \la +2$ we have
$f(2) = 2^\nu.$ Later on, we cope with this difficulty by
considering evaluations on dense subsets, cf.~Definition~\ref{tang3}
below. This problem does not arise for monomials, so we can refer to
tangible monomials without ambiguity.
\end{rem}

\begin{lem} $\tTFunSF$ is a monoid, and $ (\tR, \tTFunSF , \tGFunSF, \nu)$ is a supertropical
\domain0.
\end{lem}
\begin{proof} For $f,g \in \tTFunSF$, the essential monomials of $fg$ are products of essential monomials
and thus tangible.  Clearly  $\tR$ is a \vdomain0, seen by
restricting Remark~\ref{functan}, and $\nu_{\tTFunSF}$ is onto, by
inspection.
\end{proof}

 Given a monomial $h = \a \la_1^{i_1}\cdots
\la_n^{i_n},$ we write $\hth$ for $ \widehat \a \la_1^{i_1}\cdots
\la_n^{i_n},$ and  for the decomposition $f = \sum_i h_i$ we write
$\htf$ for $ \sum_i \hth_i$ --  the tangible lift of $f.$

\section{Supertropical $ \mfC(\tR) $-varieties}\label{cancel}

We work over a \vsemifield0 $F = (F, \tT, \tG, \nu)$, and fix a
subset $S \subseteq F^{(n)}.$ Recall that $\tR$ denotes $\Pol(S,F)$,
$\Lau(S,F)$, or $\Ratl(S,F)$, and monomials are taken in the
appropriate context.
 In principle, we want to designate a
family $\mfA(\tR)$ of tropical algebraic subsets of $S $ with
respect to elements of $\tR$. An algebraic set then is
$\mfA(\tR)$-\textbf{irreducible} if it cannot be written as the
proper union of two $\mfA(\tR)$-algebraic sets, and
 $\mfA(\tR)$ is \textbf{Noetherian} if every descending chain of $\mfA(\tR)$-algebraic sets
 stabilizes. In this section we deal with the supertropical version.

\subsection{Supertropical algebraic sets}\label{supertrop1}

\begin{defn}
%
%
%

Take some set $S \subseteq F^{(n)}$. An element $\bfa \in \tSS$ is a
\textbf{corner root} of $f \in \tR$  if $   \htf(\htbfa)\in \tG$.
The \textbf{(affine) corner locus}
 of
$f$
 with respect to  the set $\tSS$ is
$$\tZ_\corn (f;\tSS) : = \{ \bfa \in \tSS \ds : \bfa
\text{ is a corner root of }  f \}.$$  We write $\tZ_\corn (f)$ for
$\tZ_\corn (f;F^{(n)})$.
  The \textbf{total locus}
 of $f$ is
$$\tZ (f;\tSS) : = \{ \bfa \in \tSS \ds : f(\bfa)\in \tG
\}.$$
\end{defn}

\begin{defn}\label{cong}

 The \textbf{(affine) corner
algebraic set} and the \textbf{(affine)  algebraic set} of a
non-empty subset $\mathcal I  \subseteq \tR$, with respect to the
set $\tSS$,
  are  respectively
$$\tZ_\corn (\mathcal I ;\tSS) : = \bigcap _{f\in \mathcal I } \tZ_\corn(f;\tSS), \qquad  \tZ (\mathcal I ;\tSS) : = \bigcap _{f\in \mathcal I } \tZ(f;\tSS).$$
%
  When $\tSS$ is unambiguous (usually $F^{(n)}$), we
write $\tZ_\corn(\mathcal I )$ and $\tZ(\mathcal I )$  for
$\tZ_\corn(\mathcal I ;\tSS)$ and $\tZ(\mathcal I ;\tSS)$
respectively.
\end{defn}

 \begin{example}\label{fiber} Given $\bfa = (a_1, \dots, a_n )\in F^{(n)},$ the corner algebraic
 set of the
non-empty subset $\{\la_1 +a_1, \dots, \la_n +a_n\} \subseteq \tR$
consists of all vectors $\nu$-equivalent to $\bfa$, i.e., the
$\nu$-fiber of $\bfa$, and could be considered as the $\nu$-analog
of a point. These are the minimal corner algebraic
 sets in $F^{(n)}.$
\end{example}

As usual, a \textbf{hypersurface} is the algebraic set of a single
polynomial. A \textbf{facet} of a hypersurface $X = \tZ(f)$, $f =
\sum_i h_i$ is a decomposition,
 is a maximal (with respect to inclusion) connected subset of $X$ contained in the hypersurface $\tZ(h_i + h_j)$ for some $h_i, h_j$  or $\tZ(h_i)$  (for a ghost monomial $h_i$). A \textbf{face} is a nonempty intersection of facets.
A \textbf{facet} of an algebraic set $X = \tZ(\tI) = \bigcap_{f \in \tI} \tZ(f)$ is a maximal connected subset $W \subseteq X$ contained in an intersection of facets of $\tZ(f)$, $f \in \tI.$


We want our varieties to be the irreducible algebraic sets, and
these should correspond to the irreducible congruences.
%
%
But there are subtleties that have to be dealt with. For $S \subset
F^{(n)}$ we write $S|_\tng$ for $S \cap \tT^{(n)}$, the
\textbf{\tngres } of $S$.

\begin{example}\label{int1} Let  $X_1$ be the tropical line  defined by the polynomial $\la
_1 + 1\la _2 + 1$ and  $X_2$ be the tropical curve  defined by the
polynomial  $\la _1 \la _2 + \la _1 + 0$, see Fig. \ref{fig:1}.  (This can be viewed as the
curve of the Laurent polynomial  $ \la _1^{-1} + \la _2 + 0$, which
is a flip of the tropical line, cf.~Remark~\ref{Lau}.) Then $(X_1
\cap X_2)|_\tng$ is just the segment $[0,1]$ on the $\la_1$-axis, so
we see that any segment can be obtained as a corner algebraic set.
This means that we will not have irreducible algebraic sets other than points,
unless we make a serious restriction on the algebraic sets that we
admit!
\end{example}

\begin{figure}[h]
\setlength{\unitlength}{0.5cm}
\begin{picture}(10,9)(0,0)
\grid
\Thicklines
\put(1.5,5){\line(1,0){4.5}}
\put(6,1.5){\line(0,1){3.5}}
\put(6,5){\line(1,1){2.5}}
{\red \put(6,5){\line(1,0){2.5}}
\put(5,5){\line(0,-1){3.5}}
\put(5,5){\line(-1,1){3.5}}
}
\end{picture} \vskip -4mm
\caption{}\label{fig:1}
\end{figure}

Likewise, any congruence defines its  algebraic set:

\begin{defn}
An element $\bfa \in \tSS$ is a \textbf{corner root}   of a pair
$(f,g) $ (for $f,g \in \tR$) modulo a congruence~$\Cong$,  if $
\htf(\htbfa) \equiv \htg(\htbfa) \in \tG$. The \textbf{(affine)
corner locus}
 of
$f\in \tR$
 with respect to  the set $\tSS$, modulo~$\Cong$, is
$$\tZ_\corn ((f,g);\tSS)_{\Cong}  : = \{ \bfa \in \tSS \ds : \bfa
\text{ is a corner root of }( f,  g) \}.$$  We write $\tZ_\corn
(f)_{\Cong}$ for $\tZ_\corn (f;F^{(n)})_{\Cong}$.
  The \textbf{total locus}
 of $(f,g)$, modulo~$\Cong$, is
$$\tZ ((f,g);\tSS)_{\Cong} : = \{ \bfa \in \tSS \ds : f(\bfa) \equiv g(\bfa)\in \tG
\}.$$
\end{defn}

\begin{defn}\label{modcong}
 The \textbf{(affine) corner
algebraic set} and the \textbf{(affine)  algebraic set} of a
non-empty subset $A \subseteq  \tR  \times
\tR$   modulo a congruence $\Cong$, with respect
to the set $\tSS$,
  are  respectively
$$\tZ_\corn (A;\tSS)_{\Cong}: = \bigcap _{(f,g) \in A} \tZ_\corn((f,g);\tSS)_{\Cong}, \qquad  \tZ (A;\tSS)_{\Cong} : = \bigcap _{(f,g)\in A} \tZ((f,g);\tSS)_{\Cong}.$$
%
  When $\tSS$ is unambiguous (usually $F^{(n)}$), we
write $\tZ_\corn(A)_{\Cong}$ and $\tZ(A)_{\Cong}$  for
$\tZ_\corn(A;\tSS)_{\Cong}$ and $\tZ(A;\tSS)_{\Cong}$ respectively.
\end{defn}

Note that any   (corner) algebraic set of a set $A \subseteq \tR$ is
a (corner) algebraic set of $A$ modulo the trivial congruence. Thus
Definition~\ref{modcong} encompasses Definition~\ref{cong}.

\begin{defn}\label{modcong1}
Given a family
$ \mfC (\tR)$ of congruences on $\tR$, we define a  $ \mfC (\tR) $-\textbf{(corner)
algebraic set} to be a  (corner) algebraic set modulo some
congruence in $ \mfC(\tR) $. A $ \mfC(\tR) $-(corner) algebraic set is
 $ \mfC(\tR) $-\textbf{irreducible} if it cannot be written as the union
 of two $ \mfC(\tR) $-(corner) algebraic sets. A  $ \mfC(\tR)
 $- \textbf{(corner) variety} is an irreducible $ \mfC(\tR) $-(corner) algebraic
 set.
\end{defn}

The $ \mfC (\tR)$-varieties are the basis for tropical geometry,
under the appropriate choice of $ \mfC (\tR)$.

\subsection{The Zariski topology}$ $

We continue with the appropriate version of the Zariski topology.
Each essential monomial of a polynomial defines an open set
comprised of the points at which it dominates the other monomials.
Let us formalize this notion.

%
%
\begin{defn}
For any decomposition $f = \sum_{i} h_i$ of a polynomial $f \in
\mcR$, define the \textbf{component}~ $D_{f,i}$ to be
\begin{equation}\label{Eq:Dif}
D_{f,i} := \{\bfa  \in S : \htf(\htbfa) = \hth_ i (\htbfa) \}.
    \end{equation}
    A component  $D_{f,i}$ is \textbf{tangible} if the monomial
    $h_i$ is tangible, i.e., the $h_ i (\htbfa)\in \tT $ for all $\bfa \in D_{f,i}.$

 We call ${h_ i}$ the \textbf{dominant summand} of $f$ on
$D_{f, i}$. The \textbf{weak topology} is comprised of the tangible
open sets generated by the components.
\end{defn}

(Note that these are open, because the dominant monomials change at
the closure.) But this is not the topology that we want to work
with, since open sets need not be dense.

\begin{defn}
We define the \textbf{principal corner open sets} to be
$$\tD_{\corn}(f;S) = S \setminus \tZ_{\corn}(f;S) = \bigcup_{i \in I}
 D_{f,i},$$ taken over all components

Put another way, $$
\begin{array}{rcl}
\tD(f;S)_{\corn} & = &  \{ \bfa \in \tSS: \htf(\htbfa)= \hth_i(\htbfa) \text{
for some  unique monomial $h_i$ of
} f\}.
\end{array} $$
 \end{defn}
 The principal corner open sets form a base for a
topology
 on $\tSS$, which we call the \textbf{corner Zariski topology},
 whose closed sets are
affine corner algebraic sets.

We quote \cite[Proposition~9.4]{IKR4}:
\begin{prop}\label{top0} The intersection of two principal corner open sets contains
 a nonempty principal corner open set. Hence, the
principal corner open sets form a base of a topology on $\mcR$, in
which every open set is dense.
\end{prop}

From now on, we use this topology, and its relative topology on any
subset $S$ of $F^{(n)}$.

\subsection{Tangible polynomials}\label{supertrop2}$ $

The naive choice for tangibles, $ \stTFunSF$, cf. Remark
\ref{functan}, would not include polynomials (except tangible
constants) since they all have corner roots and thus are not in $
\stTFunSF$. The Zariski topology gives us a better $\nu$-structure
for polynomials, which matches Definition~\ref{tanright}.

\begin{defn}\label{tang3} A function $ f\in Fun(S,R)$ is \textbf{tangible} over $S$
if  $\{ \bfa \in S : f(\htbfa)\in \tT\}$ is dense under the
relative Zariski topology on $S$ induced from $\tR$. $\tTFunSF
$ is the set of tangible polynomials of $\tR$, and $$\tGFunSF := \{
f \in \tR : f(\htbfa) \in \tG \text{ for all } \bfa \in S \; \}$$ is the
set of ghost elements of $\tR$.
\end{defn}

\begin{remark}\label{rmk:4.12}
Any polynomial $f\in  \tTFunSF$ is  tangible over $F^{(n)}$.
Conversely, when $f$ is tangible over~$F^{(n)}$, its essential
monomials all must have tangible coefficients, since any
quasi-essential monomial is dominated by the other monomials on a
dense set.
\end{remark}

The next observation explains why we can
exclude the inessential monomials (even when quasi-essential) in the
shell of $f$.

\begin{lem}\label{goodshell} Suppose $f =  \sum_i h_i \in \tR$, written as a sum of
monomials,  and, for $\bfa \in S$, let $$\text{$f_\bfa :=  \sum_i
\big \{h_i: h_i$ is essential at $\bfa \big \}.$}$$ Then $f(\bfa)
=f_\bfa(\bfa)$ in either of the following cases:
\begin{enumerate} \eroman \item $\bfa$ is an interior point  in the
$\nu$-topology, or \pSkip \item $f_\bfa(\bfa)\in \tG.$
\end{enumerate}
\end{lem}
\begin{proof} (i) Otherwise, $f(\bfa) \ne f_\bfa(\bfa)$ would imply
$f_\bfa(\bfa)$ would be tangible, i.e., there would be only one
monomial $h_i$ essential at $\bfa$, for which $h_i(\bfa)= f(\bfa)$.
But the assumption that $\bfa$ is an interior point implies that any
quasi-essential monomial $h$ at $\bfa$ satisfies $h(\bfb)>_\nu
h_i(\bfb)$ for some $\bfb$ in a neighborhood of $\bfa,$ and taking
$\bfb$ near enough to $\bfa$ yields $f(\bfb) = h_i(\bfb) \le_\nu
h(\bfb),$ contrary to the definition of quasi-essential.

(ii) Either $f(\bfa) = f_\bfa(\bfa)$ or   $f(\bfa) =
f_\bfa(\bfa)^\nu = f_\bfa(\bfa).$
\end{proof}

\section{The coordinate \semiring0}\label{coor}$ $

We return from tropical geometry to algebra via the coordinate
\semiring0, just as in classical algebraic geometry.

\begin{defn}\label{coord2} The \textbf{coordinate \semiring0}
of an affine   algebraic set $X \subseteq F^{(n)}$, denoted~$F[X],$
is  the  image of the \semiring0 map $\Pol(F^{(n)},F)\to\Fun(X,F)$
given by the natural restriction   $f \mapsto f|_X$. The
\textbf{Laurent coordinate \semiring0} $F[X^\pm]$ is the image of
$\Lau(F^{(n)},F)$ in $\Fun(X,F)$. (Similarly, we could define
$F[X]_{\rat}$ to be image of $\Ratl(F^{(n)},F)$ in $\Fun(X,F).$)
\end{defn}



\begin{prop}\label{goodshell1} Any polynomial $f \in F[X]$ has the
same image on the interior of $X$ as its shell in $\Fun(X,F).$
\end{prop}
\begin{proof} We use Lemma~\ref{goodshell} to remove all the
inessential monomials.
\end{proof}

We have a $\nu$-structure induced by functions. Define
$F[X]_{\operatorname{tng}}$  to be those polynomials which are
tangible in the sense of Definition~\ref{tang3}, and
 $F[X]_{\operatorname{gh}}$ to be the restriction of
$\tGFunSF$  to $X$.

\begin{lem} $F[X]_{\operatorname{tng}}$ is a monoid, and $ (F[X], F[X]_{\operatorname{tng}} , F[X]_{\operatorname{gh}}, \nu)$ is a $\nu$-\semiring0.
Likewise for $F[X^\pm]$  and $F[X]_{\rat}$. \end{lem}
\begin{proof} For $f,g \in F[X]_{\operatorname{tng}}$,  $\{ \bfa \in S : f(\htbfa),g(\htbfa)\in \tT\}$ is the intersection of
two dense sets and thus is dense, implying $fg \in
F[X]_{\operatorname{tng}}$. The last assertion is clear by
restricting Remark~\ref{functan}.
\end{proof}

\begin{example} $F[X]_{\operatorname{tng}}$ is not supertropical,
since $\nu$ no longer is onto. Indeed,  let $X$ be the
supertropical line, e.g., consider the algebraic set of the polynomial $f =
\la_1 + \la _2 + 0.$ The restriction of $f$ to $X$ is ghost by
definition, and any tangible lift $\htf$ would have to include
either $\la_1 + \la _2 $ or $0$, seen by considering the vertical
and horizontal rays. But then the (tangible) diagonal ray must include $\la_1 +
\la _2 $ or $0 + \la _i$ for $i = 1$ or $i=2$, and then $\htf$
produces a ghost value  on one of the rays, contrary to it being
tangible on $X$.

Note that  $ \la_1^2 + \la _2^2 + 0$ ($= f^2$ as a function) does have the tangible lift
$\la_1 \la_2 + 0$ on $X$. Likewise, $f$ has the tangible lift
$\la_1^{\frac 12} + \la_2^{\frac 12} + 0$ on $F[X]_{\rat}$.
\end{example}

\begin{remark}\label{rmk:XinYcord}
When $X \subset Y$ we have a natural homomorphism  $ \Fun(Y,F)\to
\Fun(X,F)$ obtained by restricting the domain of the function from
$Y$ to $X$. This induces  natural homomorphisms  $ F[Y]\to F[X]$,
  $ F[Y^\pm]\to F[X^\pm],$ and $F[Y]_{\rat}\to F[X]_{\rat}.$
\end{remark}

  The restriction map gives rise to a congruence $\Cong$ on  $ F[Y]$, for which
  $ F[X] \cong  F[Y]/\Cong$.
  Conversely, we say that a congruence $\Cong$ on  $ F[Y]$ is
  \textbf{geometric} if $  F[Y]/\Cong \cong F[X]$ for some $X
  \subseteq Y.$ Then we have a 1:1 correspondence between geometric
  congruences and coordinate  \semirings0.

\begin{remark} Any geometric congruence $\Cong$ on $F[X]$
is a power-cancellative congruence which is cancellative with
respect to the tangible polynomials.
  \end{remark}

\begin{defn}\label{geomcong}
Given a subset $X \subset S$, the congruence $\Cong_X$ on $F[X]$,
called the \textbf{congruence of} $X$,
  is defined by the relation $$f \equiv_X g \qquad \text{iff} \qquad f(\bfa) = g(\bfa) \text{ for all } \bfa \in X,$$
  which we call a \textbf{polynomial relation} on $X$.
\end{defn}

  \begin{example}\label{exmpBinomial}
   If a monomial $h_i$ dominates
  $f$ and a tangible monomial $h'_j$ dominates~$g$ on some subset $W$ of $X$, then the polynomial
  relation on that subset
  is given by $h_i(\bfa) = h'_j(\bfa)$ for all $\bfa \in W,$
   which can be viewed as a Laurent relation
  $ \frac{h_i}{h'_j}(\bfa) = \fone$ on $W,$ and can be used in $W$ to
  eliminate any one variable appearing nontrivially.
  \end{example}

\begin{remark} $X$ is an algebraic set precisely when $\Cong_X$ is a geometric congruence.
Thus we have a 1:1 correspondence between algebraic sets and geometric congruences.
  \end{remark}
   An example of a non-geometric congruence:
\begin{example}
Define the congruence  $\Cong_1$  on $F[X]$  generated by
 $$ \text{$ \Cong_1 := \{(f,g): $  $f$ and $g$ both lack constant terms$\}.$}$$ Then the images in
$F[X]/\Cong_1$ of all constants are distinct, and we also have the
classes of $\la + \a$ for each~$\a \in F$. $F[X]/\Cong_1$ contains
one more class, comprised of all polynomials lacking constant terms.

Next, define  the congruence $\Cong_2$  on $F[X]$  generated by
$\Cong_1$ and $\{ (\a, \beta): \a, \beta \in F\}$. Then  $
F[X]/\Cong_2$ has only three elements: The classes of $\fone$,
$\la,$ and $\la +\fone$,
\end{example}

\begin{example}\label{phony}
Define $\Cong$  on $F[X]$ to be the congruence generated by some
pair $(f,g)$ where
  $f$ and $g$ both have the same leading monomial in $\la _1$.
 For example, take $(f,g) = (\la _1^2 + \la _2, \la _1^2 + \la_2\la_3) \in \Cong$. Then $\Cong$
 restricts to the trivial equivalence wherever $\la_2, \dots, \la_n$
 are specialized to elements small enough in relation to $\la _1.$
\end{example}

%

 The familiar
correspondence between  coordinate \semirings0 and algebraic sets is
discussed in \cite{IKR4}.  Lemma~\ref{cancext} shows that the
coordinate \semirings0 all are
{$\nu$-\domains0}. We want to single out those coordinate
\semirings0 corresponding to algebraic sets that have tropical
significance, and use these to define tropical dimension. This is an
extremely delicate issue, since various natural candidates for
tropical varieties fail to satisfy the celebrated ``balancing
condition'' \cite{IMS}.  For example, as is well known, the
intersection of the (standard) tropical lines defined by the
polynomials $\la _1 + \la _2 + 0$ and $\la _1 + \la _2 + a$ for
$a>0$ is just the ray given by $\la_1 = \la _2$ starting at $(a,a)$.
Thus, if we were to define a variety as the intersection of tropical
curves, we would have to cope with line segments of arbitrary
length. Likewise, the intersection of the curves defined by $\la_1 +
\la_2^k+0$  over $k \in \Net$ is just two perpendicular rays. So we
need   conditions to identify such degeneracies, preferably in terms
of polynomials.

\begin{rem}\label{adm0}
By Proposition~\ref{multi12}, if two monomials agree on a dense
subset of $X$, then they agree on~$X$. It follows that if two
polynomials $f$ and $g$ agree on a dense subset of $X$ then $f(\bfa)
\nucong g(\bfa)$ for all $\bfa \in X;$ in other words, their only
difference is in being ghost or not.
\end{rem}
\begin{defn}\label{admiss1} Two polynomials $f$ and $g$ \textbf{essentially agree} on  $X\subseteq F^{(n)}$ if   there is an open dense
subset   $U$ of   $X$ (in the relative topology obtained from the
 Zariski topology) for which $f|_U  = g|_U.$
The coordinate \semiring0 $F[X]$ is \textbf{\admissible} if it is a
\vdomain0 satisfying the following condition:
\begin{itemize}
  \item[$\divideontimes$] Any two
polynomials $f$ and $g$ that essentially agree on $X$ are equal.
\end{itemize}\end{defn}

 We now get to our main objective.

%

\begin{defn}\label{admiss11}  An  \textbf{admissible (corner) algebraic set} is a (corner) algebraic
set whose coordinate \semiring0\ is an \admissible\ $\nu$-\domain0.

$\mfCA$ is the set of geometric congruences corresponding to
admissible (corner) algebraic sets.
 \end{defn}

\begin{example}\label{badshell0}  Consider  the  surface  $X := \tZ_\corn(f)$ of the polynomial
$f = \lm_1 + \lm_2 + \lm_3 +0$ in $F^{(3)}$, where we erase the
facets contained in the hyperplanes determined by  $\lm_1=\lm_2$ and
$\lm_3=0$, and take the closure.
  Then the functions $\lm_1+ \lm_2 $ and $\lm_3 + 0$ are the same on all
  points
  except $(\a,\a,\a)$ for $\a >_\nu 0$ and  $(0,\bt,0)$ for $\bt <_\nu 0$, where $\a, \bt \in \tT$, for which one side
  is
 ghost and the other tangible. Thus, $\lm_1+ \lm_2 $ and $\lm_3 + 0$ essentially agree on $X$, and $X$ is not admissible.
  \end{example}

\begin{example}\label{badshell1}
Let $X$ be the hypersurface defined by the tangible polynomial $f = \sum_i
h_i,$ written as a sum of at least 3 monomials, and let $X'$ be
obtained by erasing the set   $\{ \bfa \in X \ds: h_1(\bfa) =
h_2(\bfa)>_\nu h_i(\bfa), i \ge 3\}$ and taking the closure.
(Renumbering the $h_i$ if necessary, we may assume that this set is
nonempty.)  Let $f_k = \sum _{i\ne k} h_i,$ for $k = 1,2$. Then
$f_k$ is ghost on every facet of $X'$ except those defined by
$h_i+h_k$, and furthermore $f_1|_{X'} \nucong f_2|_{X'}$ since, by
definition,
  we are left with segments in which some $h_i$ dominate for $h_i \ne
  h_1,h_2.$ Hence, $f_1 $ and $f_2$ essentially agree on $X$, and
$X$ is not admissible.
Note that $f_1f_2$ is ghost on $X'$.
  \end{example}

In this way, we exclude intersections of algebraic set in which a
facet is eliminated. We also must cope with examples such as the
intersection of the planar curves defined by $\la_1+\la_2+0$ and
$\la_1+\la_2 +1.$

\begin{example}\label{badshell2}
Let $X_{a_i}$ be the curve defined by the polynomial $f = \la_1 +
\la_2 + a_i,$ for $a \in F.$ If $a_1 <_\nu a_2$ then the \tngres \
of the intersection $X:=X_{a_1}\cap X_ {a_2}$ is the ray $\{ (b,b):
b \ge_\nu a_2, b \in \tT \} \subset \tT^{(2)}.$ The functions $f_1 =
\la_1$ and $f_2 = \la_1 + a_2$ agree for every $b
>_\nu a_2$ in $\tT,$ but $f_1((a_2,\udscr)) = a_2$ whereas $f_2((a_2,\udscr)) = a_2^\nu.$
 Hence, $X$ is not admissible.  \end{example}

 Clearly admissibility can be checked locally, i.e., at each
 neighborhood of each point $\bfa \in S$, so the next observation is the
 key.

\begin{prop}\label{admis}  Any
hypersurface defined by a  tangible polynomial is an \admissible\
algebraic set.

\end{prop}
\begin{proof}
We need to show that the coordinate \semiring0 of a hypersurface $X
:= \tZ_\corn(f)$ defined by the polynomial $f = \sum_i f_i$ is
\admissible. Suppose that polynomials $g_1 = \sum_j h'_j$ and  $g_2=
\sum_k h''_k$ essentially agree on  $X$. We want to check that they
agree on any given point $\bfa$ of $X$.  By hypothesis they agree on
some dense subset of some small open set $U \subset X$ whose closure
contains $\bfa.$ We replace $g_1$ and $g_2$ by their essential parts
on $U$. Since $g_1(\bfa) \nucong g_2(\bfa)$ by Remark \ref{adm0}, we
are done unless say $g_1(\bfa) \in \tT$ whereas $ g_2(\bfa)\in \tG$.
Thus $g_1$ has only one dominant monomial $h'_j$ at $\bfa$, whereas
$ g_2(\bfa)$ has at least two essential monomials $h''_1, h''_2,
\dots, h''_t$ at $\bfa$. By hypothesis, there are facets $C_1, C_2,
\dots, C_t$ of $U$, defined by binomials of $f$, for which
$h'_j|_{C_k}= h''_k|_{C_k}$, $k= 1,2, \dots, t.$ It is convenient to
work with Laurent polynomials, cf.~Remark~\ref{Lau},  since then we
can divide out by some given $h'_i$ and assume that $h'_i$ is the
constant monomial ~$\fone$.

Likewise, we may normalize $f$ as a Laurent polynomial to assume
that one of the essential monomials of $f$ at $\bfa$ is $\fone$, and $C_1$
is given by $f_1+\fone$.  Then $C_2$ is given by $f_2+f_3$ where $f_2
\ne \fone,$  and $f_1+f_2$ defines another facet, on which $\fone \ne
h''_k,$ a contradiction.
\end{proof}

In particular, the coordinate \semiring0 of a tropical line is
\admissible. The proposition fails for non-tangible polynomials,
since the neighborhood of a point might not have enough components
to get the contradiction in the previous proof.

\begin{example}\label{badshell02} Let $f = \la^2 + a^\nu\la + ab$ for $b<_\nu 1$, for whose
\algset\ $X = \tZ(f)$ the tangible part is the interval $X|_\tng =
[b,a].$ Then $\la$ and $\la + b$ agree on $X \setminus \{ b\}$ but
not on $b$, since $b +b = b^\nu \ne b.$  Hence, $X$ is not
admissible.
 \end{example}

\begin{example}\label{badshell}
Here is an example of  how a monomial can be quasi-essential at one
point of a hypersurface~ $X$ but essential at another portion of
$X$.

Let $X$ be the hypersurface defined by the polynomial $\la_1 + \la
_2 + 2\la_3 + \la_3 \la_4,$ and let  $f = 2 \la _1^2  + 2 \la _2^2 +
\la_1 \la_2 \la_4.$ When $\la_ 3$ takes a   small value with respect
to the substitutions of $\la_1, \la_2, $ and $\la _4,$ $X$ becomes
the \algset\  of $\la _1 + \la _2,$   for which the monomial $\la_1
\la_2 \la _4$ can be essential in $f$. But when $\la_3$ takes on a
  large value, with
respect to the substitutions of $\la_1, \la_2, $ $X$ becomes the
\algset\   of $2 + \la _4$, i.e., $\la _4 = 2,$ where $\la _1 \la_2
\la _4$ is only quasi-essential in $f$.
  \end{example}

\begin{figure}[h]
\setlength{\unitlength}{0.5cm}
\begin{picture}(10,9)(0,0)
\grid
\thicklines
\put(1.5,5){\line(1,0){2.5}}
\put(6,5){\line(1,0){2.5}}
\put(5,6){\line(0,1){2.5}}
\put(5,1.5){\line(0,1){2.5}}
\pspolygon[fillstyle=none,fillcolor=lgray](2,2.5)(2.5,2)(3,2.5)(2.5,3)
\put(4,0){(a)}
\end{picture}
\begin{picture}(10,7)(0,0)
\grid
\thicklines
\put(1.5,5){\line(1,0){2.5}}
\put(6,5){\line(1,0){2.5}}
\put(5,6){\line(0,1){2.5}}
\put(5,1.5){\line(0,1){2.5}}
\pspolygon[fillstyle=solid,fillcolor=lgray](2,2.5)(2.5,2)(3,2.5)(2.5,3)
\put(4,0){(b)}
\end{picture}
\begin{picture}(10,9)(0,0)
\grid
\thicklines
\put(1.5,5){\line(1,0){7}}
\put(5,1.5){\line(0,1){7}}
\pspolygon[fillstyle=none,fillcolor=lgray](1.75,2.5)(2.5,1.75)(3.25,2.5)(2.5,3.25)
\put(4,0){(c)}
\end{picture}

\caption{}\label{fig:2}
\end{figure}

 \begin{example} We consider some familiar examples from tropical geometry,
 viewed in the supertropical context.
  \begin{enumerate} \eroman \item
  Let $f_k = \la_1 +\la_2^k +0.$ Its corner locus  $X$  is \admissible, by
  Proposition~\ref{admis}. \pSkip

  \item In (i), Obtain $Y$ by erasing the ray given by $\la_1 =\la _2^k$
      from $X $.
  (Note that $Y =  \tZ_{\corn} (K),$ where $K = \{f_1, f_2\}$.) The polynomial
  $\la_1+\la_2$ takes on the value $0$ at each point of $X|_\tng$ except
  $(0,0)$, where it takes on the value $0^\nu.$ Thus,
  $\la_1+\la_2$ essentially agrees with the constant function $0$ on~  $X|_\tng$ but they do
  not agree on all of $X$, so $Y$ is not
  \admissible. \pSkip

\item $f = \la_1 ^2 + 3\la _1 + \la _2^2 + 4\la _1 \la _2 +5.$ Specializing $\la _2$ to some small value sends the \algset\  of~$f$ to an \algset\  in
which $\la_1 = 3$ or $\la _1 = 2;$ i.e., the \algset\  has become
disconnected and reducible. The same effect can be applied to
tropical elliptic curves. \pSkip

  \item $f = \la_1^2  \la_2^2 + \la_1^2 + \la_2^2 + 0 +
  1\la_1\la_2.$ The \tngres \ of its corner locus   is a square with a ray emanating from
  each vertex
  in the appropriate direction.  (See Fig. \ref{fig:2}(a).)
  \pSkip

  \item $f = \la_1^2  \la_2^2 + \la_1^2 + \la_2^2 + 0 + 1^\nu\la_1\la_2.$ The \tngres \ of its locus    is a filled square with a ray emanating from
  each vertex
  in the appropriate direction. (See Fig. \ref{fig:2}(b).)
\pSkip

 \item $f = \la_1^3 \la_2^3 +  1\la_1^2  1\la_2^2 +  \la_1^2  \la_2  +  1\la_1  \la_2^2  + \la_1^3 + \la_2^3  + 1\la_1\la_2+ 0.$  The \tngres \ of its corner locus   is similar to that in~(iv),  but with the four rays continuing
  inside the square, meeting at the origin. (See Fig \ref{fig:2}(c).) Thus, one could start
  with a tropical curve (a \admissible\ coordinate \semiring0), erase a few
  lines, and still have a tropical curve.  \pSkip

\item $f = \la_1^2  \la_2^2 + 2\la_1^2 \la_2 + 2\la_1\la_2^2  +
  3\la_1\la_2 +\la_1 + \la_2  +2 $ and $g = \la_1^2  \la_2^2 + 2\la_1^2 \la_2 + 2\la_1\la_2^2  + \la_2  +2
  $, yielding
  two quartic tropical curves. Although we get $g$ by erasing two
  monomials of $f$, the curve of $f$ is not contained in the curve
  of $g$. \pSkip

  \item $f = \lm_1 + \lm_2 + \lm_3+0$, and then we erase the facets defined by $\lm_1=\lm_2$ and $\lm_3=0$.
  The functions $\lm_1=\lm_2$ and $\lm_3=0$ are the same on all points
  except $(\al,\al,\al)$, $(\bt,\bt,\bt)$, and $(0,\bt,0),$ where one side can
  be ghost and the other tangible.
  \end{enumerate}

  \end{example}

Let us generalize (vi). \begin{prop}\label{admp} Suppose  $X =
\mcZ_{\corn}({f })$ where
  $f = \sum_i h_i \subset F[\lm_1, \lm_2]$ is essential, and
  erase the tangible  facet given by $\mcZ_{\corn}(h_1 + h_2).$ The ensuing
curve is not \admissible. \end{prop}
\begin{proof}
This was considered in Example~\ref{badshell1}. We claim that the
polynomial $(\sum _{i\ne 1} h_i)(\sum _{i\ne 2} h_i)$ agrees with
$g: = f  \sum _{i\ne 1,2} h_i$
  on a dense subset of $X$. This is seen from
  by considering each segment in turn, defined by $h_i+h_j$. If
$i,j >2$ the assertion is obvious, so we may assume that $i>2 \ge
j.$ Then on the interior of this segment we have $(h_i + h_j) h_i$
on both sides, proving the claim.

On the other hand, at the intersection in which $h_1,h_2$ agree but
not with any other $h_i$ we have $h_1(\bfa),h_2(\bfa)$ tangible, but
not $h_1(\bfa)+h_2(\bfa)$.
\end{proof}
%
%
%

\begin{rem}\label{admcor}
 An  \admissible\ \algset\ $X$ is  $\mfCA$-irreducible  iff the corresponding
 geometric congruence  is
$\mfCA$-irreducible.\end{rem}

As opposed to the classical situation, a reducible \algset \  can be
the union of irreducible \algsets\  in several different ways
(because of non-unique factorization), and thus a congruence can
be the intersection of irreducible congruences in several
different ways.

\begin{figure}[h]
\setlength{\unitlength}{0.5cm}
\begin{picture}(10,9)(0,0)
\grid
\Thicklines
\put(1.5,5){\line(1,0){3.5}}
\put(5,1.5){\line(0,1){3.5}}
\put(5,5){\line(1,1){3.5}}
{\red
\put(5,5){\line(0,1){3.5}}
\put(5,5){\line(1,0){3.5}}
\put(5,5){\line(-1,-1){3.5}} }
\put(4,0){(a)}
\end{picture}
\begin{picture}(10,7)(0,0)
\grid
\Thicklines
\put(1.5,1.5){\line(1,1){7}}
{\red \put(5,1.5){\line(0,1){7}}}
{\blue \put(1.2,5){\line(1,0){7.2}}}
\put(4,0){(b)}
\end{picture}

\caption{}\label{fig:3}
\end{figure}

\begin{example}\label{irredundant} $ $
\begin{enumerate}
    \item  The \algset \  $\mcZ_{\corn}((\la _1 + \la _2 + 0 )(\la _1 \la _2 + \la _1\la _2 + 0\la _2))$
  can be viewed as the union of the tropical line $\mcZ_{\corn}(\la _1 + \la _2 + 0)$
    and  conic $\mcZ_{\corn}( \la _1 \la _2 + \la _1\la _2 + 0\la _2)$ (see Fig. \ref{fig:3}(a)),
    as well as the three curves $\mcZ_{\corn}(\la _1 +0)$, $\mcZ_{\corn}(\la _1  +0)$, and $\mcZ_{\corn}(\la _1 +\la
    _2)$ (see Fig. \ref{fig:3}(b)).\pSkip

    \item  Although in (i), we could say that the two
    decompositions differ at the multiplicity of the point
    $(0,0)$, and thus could be detected in the layered congruence,
    Sheiner \cite[Example~5.7]{Erez1} found the following example in which
    even the multiplicities match:
   $$  \begin{array}{rr}
    (\la_2 + \la_1 + \la_1^2
+ (-1)\la_1^3)(\la_2 + 0 + \la_1^2  + (-2)\la_1^4
)  & = \\ [2mm] (\la_2 + \la_1 + \la_1^2 + (-2)\la_1^4
)(\la_2 + 0 + \la_1^2+ (-1)\la_1^3 ).
   \end{array} $$
\end{enumerate}
\end{example}

 So far we have two basic ways of initiating a homomorphism on a
coordinate \semiring0: Either restrict its \algset, or put in new
relations among the indeterminates of $\La.$ By \textbf{binomial
relation} we mean a relation of the form $h|_W = h'|_W$, where
$h,h'$ are different monomials and $W \subseteq X$ is nonempty.

\begin{lem}\label{multi4} Suppose $X\subset Y$ are \algsets.
Then the induced map $\Phi:F[Y] \to F[X]$ involves an extra
binomial relation on each facet of $Y$   not contained in a facet
of  $X$.
\end{lem}
\begin{proof}  Write
  $ F[X] \cong  F[Y]/\Cong$. On each facet of $Y$ we have some
pair $(f,g) \in \Cong$ and we take their dominant monomials $(f_i,
g_j)$ on this facet. Then $(f_i,g_j)$ is the extra binomial
relation that we want, and we are done unless always $f_i = g_j$,
which means that $\Phi$ is the identity on our facet of~$Y$, which
then is embedded in a facet of  $X$.
\end{proof}

\begin{lem}\label{multi4} Suppose $X\subset Y$ are \algset s for
which $F[Y]$ is obtained from $F[X]$ by adjoining one polynomial
relation $f=g$. Then  this polynomial relation arises from a
binomial relation  that dominates~$Y$.
\end{lem}

\begin{lem}\label{multi50} Suppose $F[X]$ is an admissible coordinate \semiring0 which
is defined by a set of polynomials $f_1, \dots, f_m$, and two of
these polynomial functions $f_1$ and $f_2$ coalesce at an interior
point $\bfa$ of some facet  $W$ of $X$. Then $f_1$ and $f_2$
agree on all of the facet $W$.
\end{lem}
\begin{proof} They agree via their leading monomials on some open subset containing $\bfa$ and
thus on all rays emanating from $\bfa$ in $W$, by Proposition
~\ref{multi12}. Suppose some other monomial $h_1$ of $f_1$ dominates
them elsewhere on $W$. Then $h_1(\bfa') \nucong f_1(\bfa') \nucong
f_2(\bfa') $ at some point $\bfa' \in W$, which is by definition on
the boundary of $W$. This would mean $ f_1(\bfa') = f_2(\bfa')^\nu$,
contradicting admissibility unless $f_2(\bfa') \in \tG,$ i.e., $f_2$
has a monomial $h_2$ such that $h_1(\bfa') \nucong h_2(\bfa'). $
Continue on an open neighborhood of $\bfa'$, and apply this argument
throughout $W$.
\end{proof}

\section{Dimensions of admissible corner varieties}\label{admiss1}

Binomials play a key role in defining corner \algsets, since
 corner \algsets\ are defined ``piecewise'' by binomials.
 Localizing $F[X]$ at the tangible monomials enables us to pass to the
 Laurent coordinate \semiring0 $F[X^\pm],$ which then is viewed
 inside $F[X]_{\rat}.$

\begin{prop}
 \label{multi3} The subset of any \algset\  defined by a given set of
 binomials is convex (and thus connected).
\end{prop}
\begin{proof} By Proposition~\ref{multi12}, for any two corner
roots of the binomial, the line joining them also consists of corner
roots (since all other monomials are dominated at these points).
\end{proof}

\begin{thm}\label{multi5} If $X\subset Y$ are $\mfC(\tR)$-corner varieties, then the induced map $\Phi:F[Y] \to F[X]$, if not~1:1, involves
  extra binomial relations which dominate $Y$.
\end{thm} \begin{proof} For any new element $(f,g)$ of a congruence,
using Proposition~\ref{multi3}, we obtain a new binomial relation $h
|_W = h'|_W$ on some facet $W$. Taking its tangible lift and localizing, we may
assume $h' = \fone$ and write $h = \a \la_1^{i_1} \cdots
\la_n^{i_n},$ for each $i_j \in \mathbb Z,$ not all 0. Reindexing
the indeterminates, we may assume that $i_n \ne 0.$ By Lemma
~\ref{multi50} we get a
 new binomial relation, which enables us to solve $$\la_n  \ds\mapsto
 \a^{-\frac 1 {i_n}}\la_1^{-\frac{i_1}{i_n}} \ds \cdots \la_{n-1}^{-\frac{i_{n-1}}{i_n}} $$
 in terms of the   indeterminates $ \la_1, \dots, \la_{n-1}$ (working in $F[X]_{\rat}$), on this facet. This
 provides an inductive procedure on each of our finitely many facets, which must
 terminate    when we eliminate all of the indeterminates in each facet.
\end{proof}

%
%
%
%
%
%
%

%

 There are several possible definitions of dimension
which can be garnered from the coordinate \semiring0. We take the
algebraic one. This is close to the approach of Perri~\cite{Per}.

\begin{defn}\label{addim} The \textbf{dimension}  $\dim X$ of an irreducible admissible corner \algset
\ (i.e., of a $\mfCA$-corner variety) is the maximal length of a
chain of $\mfCA$-subvarieties  of $X$, i.e., the   maximal length
$m$ of a chain of $\mfCA$-irreducible coordinate \semirings0 $$F[X]
= F[X_0]\to \ds\cdots \to F[X_m]$$ (where $X_m$ is the $\nu$-fiber of a point,
as in Example~\ref{fiber}).
\end{defn}

\begin{thm}\label{primelen2} If there is a chain of homomorphisms $F[\La]
\to F[X_1]\to \dots \to   F[X_m],$ where $\La = \{ \la_1, \dots,
\la_n\}$ and each $X_i$ is a  $\mfCA$-corner variety, then $m \le
n.$ Furthermore, for any such chain of maximal length, $m=n.$
\end{thm}
\begin{proof} We review the proof Theorem~\ref{multi5}, with some extra care.
These homomorphisms $F[X_{i}] \to F[X_{i+1}]$
 are obtained at each facet by new
binomial relations,  say $\fone +h$, where $h$ is a Laurent monomial.
Without loss of generality, assume that $\la_n$ appears in $h$. At
any facet for which $h$ is essential, $h$ which can be used to
eliminate the indeterminate $\la_n$ in terms of the others. We claim
that this can be done at most $n$ times at any given facet, at which
stage any polynomial is locally constant. But the constants are the
same since
  polynomials are continuous (and $X_{i}$ is connected in view of
  Proposition~\ref{admp}), so by assumption the polynomial is
constant after at most $n$ steps, which means that we cannot
continue the chain further.

The only difficulty with this argument is that some of the
reductions might be trivial, along the lines of Example~\ref{phony}.
In other words, $h$ might be dominated by $\fone$. But now we appeal to an
idea of Tal Perri in his dissertation~\cite{Per}. In order to make
$h = \a \la_1^{i_1} \cdots \la_n^{i_n}$ inessential, we must have $\a \la_1^{i_1} \cdots \la_n^{i_n} \le
_\nu \fone.$ This yields a new inequality among the indeterminates,
involving $\la_n$, which Perri calls an \textbf{order relation}.
This can only happen if $\la _n$ appears in one of the essential
monomials defining the facet, so again we can substitute for $\la
_n$ and eliminate it.

 At each step
we eliminate one more indeterminate, and so the process must terminate after $n$ steps.%
%
This proves that $m \le n.$ When $m<n,$ there remain ``free''
variables in each facet; since the facets can be viewed
locally as hypersurfaces, we conclude with
 Proposition~\ref{admis}.\end{proof}

In conclusion:

\begin{cor}
 Any chain of irreducible admissible corner \algssets \  of  $F^{(n)}$ can be refined
 to a  chain of irreducible admissible corner \algssets \ of  $F^{(n)}$ of length
 $n$.
 \end{cor}

\section{The layered approach to varieties}\label{sec:layer}$ $

Although \algset s (Definition~\ref{admiss11})
rely heavily on the use of the $\nu$-structure, applications
concerning multiple roots rely on more refined layerings, so we
briefly present the foundations for this alternative.

\subsection{Layered \domains0 and \semifields0}\label{lay1}$ $

We
recall the main example in ~\cite{IKR4}.

\begin{construction}\label{contrULD} Suppose $\tT$ is a cancellative monoid,
and ~$L$ is a  \semiring0 to be used as an index set. We define the
\textbf{ $L$-layered \domain0 $\mathscr R(L,\tT)$} (or
\textbf{layered \domain0} for short, when $L$ is understood) to be
set-theoretically $L\times \tT$, where for $k,\ell\in L,$ and
$a,b\in\tT,$  we write $\xl{a }{k} $ for $(k,a)$ and  define
multiplication componentwise, i.e.,
\begin{equation}\label{13}
\xl{a }{k} \cdot \xl{b}{\ell} = \xl{ab}{k\ell},  \end{equation} and
addition from the rules:

\begin{equation}\label{14}
\xl{a }{k}+ \xl{b}{\ell}=\begin{cases} \xl{a }{k}& \quad\text{if}\ a >  b,\\
\xl{b}{\ell}& \quad\text{if}\ a <  b,\\
\xl{a }{k+\ell}& \quad\text{if}\ a= b.\end{cases}\end{equation}
 $ \mathscr R(L,\tT)$ is   equipped  with
the \textbf{sort map}
    given by $\lv(\! \xl{a }{k})=k,$
%
%
and maps $$ \nu_{\ell,k}: (k,\tT)  \to (\ell,\tT), \qquad k \le
\ell,  \quad k,\ell \in L,$$  given by $\xl{b}{k} \mapsto
\xl{b}{\ell}$.   We define the \textbf{$\ell$-layer}  $R_\ell := \{ \xl{a}{\ell} \ds : a \in \tT \}$, and  write $R := \mathscr R(L,\tT)$ as the disjoint union $$ \mathscr R(L,\tT) = \bigcup_{\ell \in L} R_\ell.$$
We also define $e_\ell :=  \xl{\one_\tT}{\ell }.$

We write $a \nucong b$ (resp.  $a \nug b$) for $a \in R_k$ and $b \in R_\ell$,  whenever
 $\nu_{m,k}(a) = \nu_{m,\ell}( b )$ (resp. $\nu_{m,k}(a) > \nu_{m,\ell}( b )$) in $R_m$ for some $m \ge
 k,\ell$. (This notation is used generically: we write  $a \nucong b$
even when the sort transition maps  $\nu_{m,\ell}$ are notated
differently.)

\end{construction}

This construction is put into a more formal context in
\cite{IKR4,IzhakianKnebuschRowen2011CategoriesII}. In order not to
be distracted here from the impact of the algebra on geometric
considerations, we take the
 sorting \semiring0 $L$ to be a totally ordered (commutative) \semiring0, perhaps with an absorbing element $0 = 0_L$ adjoined.

 $R_1$
  is called
the set of \textbf{tangible elements} of $R$, and plays a key role
in the theory. It is convenient for~$1= 1_L$ to be the minimal sorting
index in $L$. Towards this end, for any layered \semiring0 $R$,  we
may replace $R$ by $ \bigcup _{\ell \ge 1} R_\ell,$ a sub-\semiring0
of $R$.  The tangible lift is given by   $\xl{a}{\ell} \mapsto
\xl{a}{1}.$
  We bear in mind the examples $L = \Net$,
 and  $L = \Q_{\ge 1},$  each with the usual order.
 Let $\bar L = \{1,\infty\}.$ We then have a \semiring0\
 homomorphism $L \mapsto \bar L$ sending $k \mapsto \bar k$, where $\bar 1 = 1$ and $\bar k = \infty$ for
 each $1 < k \in L.$ 

%


 \begin{defn} $\mathscr R(L,\tT)$ is a  \textbf{layered 1-\semifield0} if  $\tT$ is a group.
\end{defn}

\begin{rem}  $R_1$ is a cancellative
submonoid isomorphic to $\tT.$ Localizing $R:=\mathscr R(L,\tT)$ at
$R_1$ yields   a  layered 1-\semifield0, whose 1-layer is an ordered
group iff
 $\tT$ is an ordered monoid.
\end{rem}

\subsection{Function and polynomial  \semirings0 of layered \semirings0} $ $

 As in \cite{IKR4}, we can pass the layered structure from $R$ to $\FunSR$,
at the expense of enlarging the layering set from $L$ to $\Fun
(\tSS, L)$.

 \begin{example}\label{Layer1}  As noted in
\cite[Remark~5.4]{IzhakianKnebuschRowen2009Refined}, when $R$ is an
$L$-layered \semiring0, then $\FunSR$ is layered
  with respect to
  $\Fun (\tSS,L) $, where $\FunSR$ has the sort map $s$ given by
  $$s(f)(\bfa) = s(f(\bfa)),$$  for $\bfa := (a_1,\dots,
a_n)$ in $S.$

%

We can extend $\ge_\nu$ to a partial order on $\FunSR$ as follows:

 \begin{enumerate}\eroman

    \item $f \nucong g$ iff $f(\bfa) \nucong g(\bfa), \quad \forall \bfa
 \in S;$ \pSkip

 \item  $f >_\nu g$ iff $f(\bfa)>_\nu   g(\bfa), \quad \forall \bfa
 \in S.$
 \end{enumerate}
 Although not totally ordered, $ \FunSR$  satisfies the
weaker properties:
\medskip

 If $f > _\nu  g$,  then  $f+g = f$; $\quad 2f \nucong f$ with $s(2f) = 2s(f),$

\medskip
\noindent seen by pointwise verification.

 \end{example}

The construction and definition were generalized in
\cite{IzhakianKnebuschRowen2009Refined}, \cite{IKR4} and
\cite{IzhakianKnebuschRowen2011CategoriesII}, but we work with the
more specific case here in order to avoid further complications.

Our main interest is  in the case where $\tR: = F[\La]$ in commuting
indeterminates $\Lm = \{ \la_1, \dots, \la _n\}$ over  a layered
1-\semifield0 $F$.  We want to understand the homomorphic images of
$\tR$ by specializing certain $\la _i$ in terms of extensions of
$F$, in order to prepare the groundwork for a layered version of
affine geometry. The main idea is that in specializing $\la_1,
\dots, \la_n$ to elements of $F$, we also obtain a homomorphism
$L[\la_1, \dots, \la_n] \to L$ and thus recover the original sorting
set $L$.

%

\begin{lem} If $R = \mathscr R(L, \tT)$, then $$\Pol (\Lambda,R) = \mathscr
R(\Pol (\Lambda,L),\Pol (\Lambda,\tT)) ,$$ and $$\Laur (\Lambda,R) =
\mathscr R(\Laur (\Lambda,L),\Laur (\Lambda,\tT)) .$$
\end{lem}
\begin{proof}  The unit element of the monoid  $\Pol (\Lambda,R)$
is the constant function sending all elements to $\rone$, and its
corresponding layer  in $\Pol (\Lambda,L)$ is clearly a monoid. The
same argument holds for Laurent polynomials and rational
polynomials.
\end{proof}

%
%
%
%
%
%
%
%
%
%
%
%

 In this way, we can replace $R$ by
 $\FunSR$ in the theory described above, but at the cost of
 replacing the original sorting set $L$ by a much more complicated sorting
 set.

 The reason we used the supertropical
and not the layered structure in our definition of $ \mfC (\tR)$-variety is because of the
following sort of example.
 To ease notation, we write $a$ for $\xl{a}{1},$ and $\la$ for
 $0\la.$

\begin{example}\label{badshell00}
 Consider the corner \algset \ of the polynomial $f = \lm^2 +0$ in $F$. The function
 $g = \lm^2+\lm+0$ agrees with $f$ at all points except $a=0.$ Over
 the supertropical structure, $f(0) = 0^\nu = g(0),$  and the corner \algset \ is easily seen to
 be admissible. But the analogous property fails with respect to the
 layered structure, since $f(0) = \xl{0 }{2} $ whereas
 $g(0) = \xl{0 }{3}.$
  \end{example}

  Such an example could not interfere with the supertropical theory,
  because of Lemma~\ref{goodshell}(ii).
  Nonetheless, the layered approach enables one to cope better with
  different multiplicities of roots, and gives us the following alternative approach.

\subsection{Layered algebraic sets} $ $

Let $F = \mathscr R(L, \tT)$ be a layered 1-\semifield0.
Recall \cite[Definition~5.7]{IzhakianKnebuschRowen2009Refined}:

\begin{defn}\label{laymap} The \textbf{layering map} of a function $f\in \Fun (\tSS,F)$ is
the map $\vmap_f: \tSS\to L$ given by

$$\vmap_f(\bfa) :=
s(f(\bfa)), \qquad \forall  \bfa \in \tSS.$$
 The \textbf{layering map} of a set of functions $A \subseteq \Fun (\tSS,F)$ is
 given by
 $$\vmap_A(\bfa) := \inf _{ f \in A} \vmap_f(\bfa).$$
\end{defn}

Thus, $\vmap_A \in \Fun (\tSS,L) $. 

\begin{defn}\label{laymap1}
The \textbf{layering} $\tL(A)$ of a subset $A \subseteq \Fun
(\tSS,F)$ is the set $$ \tL(A) := \{ (\bfa, \vmap_A(\bfa)): \bfa \in \tSS \}.$$
The \textbf{\lalgset}\ $X: = X_A$ is the subset $$ X_A := \{ (\bfa,
\vmap_A(\bfa)) \in \tL(A) \ds : \vmap_A(\bfa) >1\}.$$ We write
$\underline{X}$ for the projection of $X$ onto $\tSS$, which is
 $\{\bfa \in \tSS: \vmap_A(\bfa) >1\}.$
\end{defn}

This matches our definition of \algset\, but also records the jump
in multiplicity. Thus $A$, although not always notated, is intrinsic
in the definition of $X$,
 and the second coordinate $\vmap_A(\bfa)$  plays a key role. 

\begin{defn}\label{laymap2} As in Definition~\ref{laymap}, given two \lalgset s $X= X_A$ and $Y =Y_B$, we
define $$X \vee Y = \{(\bfa, \max \{ (\vmap_A(\bfa), \vmap_B
(\bfa))\}) \ds: \bfa \in \tSS, \ \vmap_A(\bfa) >1 \text{ or }
\vmap_B (\bfa) >1 \};$$   $$ X \wedge Y = \{(\bfa, \min \{
(\vmap_A(\bfa), \vmap_B(\bfa))\})\ds : \bfa \in \tSS, \
 \vmap_A(\bfa)>1\text{ and } \vmap_B(\bfa) >1\}.$$

We say that $X \preceq Y$ if $X \wedge Y = X,$ i.e., if
$\underline{X} \subseteq \underline{Y}$ and  $\vmap_A(\bfa) \le
\vmap_B(\bfa)$ for each $\bfa \in \tSS .$
\end{defn}

\begin{rem}\label{und1} $\underline{X \vee Y }= \underline{X}\cup \underline{Y};  $  $\underline{X \wedge Y }
=\underline{X}\cap \underline{Y}.$
\end{rem}

Repeating Definition~\ref{geomcong} where now $F$ is layered, we now
call~$\Cong_X$ a  \textbf{layered congruence} of $X$, and define $
\mfCL $ to be the set of layered congruences on $\tR : = F[\Lm]$.

\begin{defn}
 A \lalgset \ $X $ is $\mfCL$-\textbf{irreducible} if it cannot be
written as $X _1 \vee X _2$ for \lalgset s $X _1 ,  X _2, \ne X,$
and
 $\mfCL$ is \textbf{Noetherian} if every descending chain of   \lalgset s
 (under~$\preceq$) stabilizes.
\end{defn}

\begin{rem}\label{admcor1}
As in Remark~\ref{admcor},
 A   \lalgset\ $X$ is  $\mfCL$-irreducible  iff the corresponding
 geometric congruence  is
$\mfCL$-irreducible.\end{rem}

\begin{example}\label{int2}  $ $
\begin{enumerate} \eroman
  \item
 Let us view Example~\ref{int1} from this
perspective.
Let  $L_\a:= X_{f_\a}$ be the tropical line  defined by the polynomials $f _\a= \la
_1 + \la _2 + \a$ and  $X:= X_{g}$ be the tropical curve  defined by the
polynomial  $g = \la _1 \la _2 + \la _2 + 0,$ and let $X_\a = L_\a \wedge X.$
  Then $(\underline{L_\a \cap X})|_{F_1 ^{(2)}}$ is just the segment $[0,\a]$ on the $\la_2$-axis, but
$\vmap_{f_\a}(0) = 2 = \vmap_{g}(\a)$ whereas $\vmap_{f_\a}(\a) = 3
= \vmap_g(0).$ In other words, when $\a < \beta$ we do not have
$X_\a \preceq X_{\beta}$ even though $\underline{X_\a} \subset
\underline{X_{\beta}}$. \pSkip

\item
 Likewise, let $L_1$ be the tropical line  defined by the polynomial $\la
_1 + \la _2 + 0$ and $X_2$ be the tropical curve  defined by the
polynomial $\la _1^2 + \la _2 + 0$. Now $(\underline{L_1 \wedge X_2})|_{F_1 ^{(2)}}$
still is the union of two rays (the lower $\la_1$ and $\la_2$ axes),
which is properly contained in $L_1$. We can get the third ray by
intersecting the hyperplanes of $\la_1 + \la _2$ and $\la_1 +0$, but
the level at $(0,0)$ is only 2, not 3. Thus, $L_1$ is irreducible
with respect to $\preceq $.
\end{enumerate}
\end{example}

Let us formalize Example~\ref{int2}(i).

\begin{rem}\label{int3}
 By definition, the layering function is constant on any
facet. Hence, if $X \preceq Y,$ every facet of $X$ is contained in
the corresponding facet of $Y$.
\end{rem}

\begin{prop}\label{Noeth}  The class of layered algebraic subsets of $F^{(n)}$
is Noetherian.
\end{prop}
\begin{proof}
There are only a finite number of facets, and each increase of the
congruence decreases the level of some facet (since by
Remark~\ref{int3} it cannot ``cut'' a facet).
\end{proof}

Note that the layered dimension of the tropical line, defined in
terms of a maximal descending chain of irreducible layered algebraic
sets would be 3, not 1, since the ray along an axis (as well as the
union of the two semi-axes) is a layered algebraic set. This
discrepancy could be resolved by further restricting our class of
congruences along the lines of \S\ref{admiss1}.

\end{document}